\documentclass[preprint,12pt]{elsarticle}
\usepackage{amssymb,amsmath, bm}
\usepackage{hyperref}
\journal{Advances in applied mathematics}

\newdefinition{Def}{Definition}[section]
\newdefinition{Rem}{Remark}[section]
\newdefinition{Exa}{Example}[section]

\newtheorem{The}{Theorem}[section]
\newtheorem{Pro}{Proposition}[section]
\newtheorem{Lem}{Lemma}[section]

\newproof{proof}{Proof}[section]

\newcommand{\Remref}[1]{Remark \ref{#1}}
\newcommand{\Tabref}[1]{Table \ref{#1}}
\newcommand{\Figref}[1]{Fig. \ref{#1}}
\newcommand{\Lemref}[1]{Lemma \ref{#1}}
\newcommand{\Secref}[1]{Section \ref{#1}}
\newcommand{\Theref}[1]{Theorem \ref{#1}}
\newcommand{\Proref}[1]{Proposition \ref{#1}}

\newcommand{\bb}[1]{\mathbb{#1}}
\newcommand{\norm}[1]{\left\|{#1}\right\|}
\newcommand{\abs}[1]{\left|{#1}\right|}
\newcommand{\scalarp}[1]{\left\langle{#1}\right\rangle}

\begin{document}
\begin{frontmatter}

\title{Weighted Thresholding and Nonlinear Approximation}

\author{Emil Solsb{\ae}k Ottosen}
\ead{emilo@math.aau.dk}
\address{Department of Mathematical Sciences, Aalborg University, Skjernvej 4, 9220 Aalborg {\O}, Denmark}

\author{Morten Nielsen}
\ead{mnielsen@math.aau.dk}
\address{Department of Mathematical Sciences, Aalborg University, Skjernvej 4, 9220 Aalborg {\O}, Denmark}

\begin{abstract}
We present a new method for performing nonlinear approximation with redundant dictionaries. The method constructs an $m-$term approximation of the signal by thresholding with respect to a weighted version of its canonical expansion coefficients, thereby accounting for dependency between the coefficients. The main result is an associated strong Jackson embedding, which provides an upper bound on the corresponding reconstruction error. To complement the theoretical results, we compare the proposed method to the pure greedy method and the Windowed-Group Lasso by denoising music signals with elements from a Gabor dictionary.
\end{abstract}


\end{frontmatter}

\section{Introduction}
Let $X$ be a Banach space equipped with a norm $\|\cdot\|_X$. We consider the problem of approximating a possibly complicated function $f\in X$ using linear combinations of simpler functions $\mathcal{D}:=\{g_k\}_{k\in \bb{N}}$. We assume $\mathcal{D}$ forms a complete dictionary for $X$ such that $\|g_k\|_X=1$, for all $k\in \bb{N}$, and $\text{span}\{g_k\}_{k\in \bb{N}}$ is dense in $X$. A natural way of performing the approximation is to construct an $m$-term approximation $f_m$ to $f$ using a linear combination of at most $m$ elements from $\mathcal{D}$ \cite{Schmidt07,Ismagilov74,Oskolkov78}. This leads us to consider the set
\begin{equation*}
\Sigma_m(\mathcal{D}):=\left\{\sum_{k\in \Delta}c_kg_k ~\Big|~\Delta\subset \bb{N},~\abs{\Delta} \leq m\right\},\quad m\in \bb{N}.
\end{equation*}
We note that $\Sigma_m(\mathcal{D})$ is nonlinear since a sum of two elements from $\Sigma_m(\mathcal{D})$ will in general need $2m$ terms in its representation by $\{g_k\}_{k\in \bb{N}}$. We measure the approximation error associated to $\Sigma_m(\mathcal{D})$ by
\begin{equation}\label{eq:errorofbestmtermapprox}
\sigma_m(f,\mathcal{D})_{X}:=\inf_{h\in \Sigma_m(\mathcal{D})}\norm{f-h}_{X},\quad f\in X.
\end{equation}
One of the main challenges of nonlinear approximation theory is to characterize the elements $f\in X$, which have a prescribed rate of approximation $\alpha>0$ \cite{Petrushev1988,Grochenig2000,DeVore1992}. This is usually done by defining an approximation space $\mathcal{A}\subseteq X$ with the property
\begin{equation}\label{eq:JacksonInequality}
\sigma_m(f,\mathcal{D})_{X}= \mathcal{O}\left(m^{-\alpha}\right),\quad \forall f \in \mathcal{A}.
\end{equation}
It is often difficult to characterize the elements of $\mathcal{A}$ directly and a standard approach is therefore to construct a simpler space $\mathcal{K}\subseteq X$ together with a continous embedding $\mathcal{K}\hookrightarrow \mathcal{A}$. The space $\mathcal{K}$ is referred to as a smoothness or sparseness space and the continuous embedding as a Jackson embedding \cite{Butzer72,DeVore1993}. 

In the special case $\{g_k\}_{k\in \bb{N}}$ forms an orthonormal basis for a Hilbert space $\mathcal{H}$, it follows from Parseval's identity that the best $m-$term approximation to $f$ is obtained by thresholding the (unique) expansion coefficients and keeping only the $m$ largest coefficients. For a redundant dictionary $\mathcal{D}$, the problem of constructing the best $m-$term approximation is in general computationally intractable \cite{Davis1997}. For this reason, various algorithms have been constructed to produce fast and good (but not necessarily the best) $m-$term approximations \cite{Darken1993, DeVore1996, Gribonval2004}. By "good approximations" we mean an algorithm $A_m:f\rightarrow f_m$ for which there exists an approximation space $\mathcal{T}\subseteq X$ with 
\begin{equation*}
\norm{f-f_m}_{X}= \mathcal{O}\left(m^{-\alpha}\right),\quad \forall f\in \mathcal{T}.
\end{equation*} 
An associated embedding of the type $\mathcal{K}\hookrightarrow \mathcal{T}$ is referred to as a strong Jackson embedding. Whereas a standard Jackson embedding only shows that the error of best $m$-term approximation decays as in \eqref{eq:JacksonInequality}, a strong Jackson embedding also provides an associated constructive algorithm with this rate of approximation \cite{Gribonval2004}.

Traditionally, the expansion coefficients are proccessed individually with an implicit assumption of independency between the coefficients \cite{DeVore1996}. However, for many signals such an assumption is not valid as the coefficients are often correlated and organised in structured sets \cite{Kowalski2009a,Kowalski2009b}. For instance, it is known that many music signals are generated by components, which are sparse in either time or frequency \cite{Plumbley2010}, resulting in sparse and structured time-frequency representations (cf. \Figref{fig:CleanNoisySpec} on page \pageref{fig:CleanNoisySpec}). In this article we present a thresholding algorithm, which incorporates a weight function to account for dependency between expansion coefficients. The main result is given in \Theref{The:mainresult} and provides a strong Jackson embedding for the proposed method. 

The idea of exploiting the structure of the expansion coefficients is somewhat similar to the one found in social sparsity \cite{Siedenburg2011,Kowalski2009c,Kowalski2013}. Social sparsity can be seen as a generalization of the classical \emph{Lasso} \cite{Tibshirani1996} (also known as basis pursuit denoising \cite{Chen2001}) where a weighted neighbourhood of each coefficient is considered for deciding whether or not to keep the coefficient. However, whereas social sparsity searches through the dictionary for sparse reconstruction coefficients, the purposed method considers weighted thresholding of the canonical frame coefficients. We compare the proposed algorithm to the Windowed-Group-Lasso (WGL) \cite{Kowalski2009b,Siedenburg2012} from social sparsity and the greedy thresholding approach  \cite{DeVore1996,Gribonval2004} from nonlinear approximation theory by denoising music signals expanded in a Gabor dictionary \cite{Grochenig2001,Christensen2016}. 

The structure of this article is as follows. In \Secref{Sec:2} we introduce the necessary tools from nonlinear approximation theory and in \Secref{Sec:3} we present the proposed algorithm and prove \Theref{The:mainresult}. In \Secref{Sec:4} we provide the link to modulation spaces and Gabor frames and in \Secref{Sec:5} we describe the numerical experiments. Finally, in \Secref{Sec:6} we give the conclusions.

\section{Elements from nonlinear approximation theory}\label{Sec:2}
In this section we define the concepts from nonlinear approximation theory that we will use throughout this article. We refer the reader to \cite{DeVore1993,DeVore98,DeVore09,Gribonval2004} for further details. For $\{a_m\}_{m\in \bb{N}}\subset \bb{C}$, we let $\{a^*_m\}_{m\in \bb{N}}$ denote a non-increasing rearrangement of $\{a_m\}_{m\in \bb{N}}$ such that $|a^*_{m}|\geq |a^*_{m+1}|$ for all $m\in \bb{N}$.
\begin{Def}
Given $\tau \in (0,\infty), q\in (0,\infty]$, we define the Lorentz space $\ell_q^\tau$ as the collection of $\left\{a_m\right\}_{m\in \bb{N}}\subset \bb{C}$ satisfying that
\begin{equation*}
\norm{\{a_m\}_{m\in \bb{N}}}_{\ell_q^\tau}:=\left\{
     \begin{array}{ll}
       \left(\sum_{m=1}^\infty \left[m^{1/\tau} \abs{a_m^*}\right]^q\frac{1}{m}\right)^{1/q},  & 0<q<\infty\\
       \sup_{m\geq 1} m^{1/\tau} \abs{a_m^*},  & q=\infty
     \end{array}
   \right.
\end{equation*}
is finite.
\end{Def}
We note that $\|\cdot\|_{\ell_\tau^\tau}=\|\cdot\|_{\ell^\tau}$ for any $\tau \in (0,\infty)$. The Lorentz spaces are rearrangement invariant (quasi-)Banach spaces (Banach spaces if $1\leq q\leq \tau<\infty$) satisfying the continuous embeddings 
\begin{equation}\label{eq:LorentzEmbeddings}
\ell_{q_1}^{\tau_1}\hookrightarrow \ell_{q_2}^{\tau_2}\quad \text{if}\quad \tau_1<\tau_2\quad \text{or if}\quad \tau_1=\tau_2 \text{ and }q_1\leq q_2,
\end{equation}
see \cite{Carro2007,DeVore98} for details. Let $\mathcal{D}:=\{g_k\}_{k\in \bb{N}}$ be a complete dictionary for a Banach space $X$ and define $\{\sigma_m(f,\mathcal{D})_X\}_{m\in \bb{N}}$ as in \eqref{eq:errorofbestmtermapprox}. We will use the following approximation spaces \cite{DeVore1993}.
\begin{Def}
Given $\alpha\in (0,\infty), q\in (0,\infty]$, we define
\begin{equation*}
\mathcal{A}_q^\alpha\left(\mathcal{D},X\right):=\left\{f\in X\Big|\norm{f}_{\mathcal{A}_q^\alpha\left(\mathcal{D},X\right)}:=\norm{\{\sigma_m(f,\mathcal{D})_X\}_{m\in \bb{N}}}_{\ell_q^{1/\alpha}}+\norm{f}_X<\infty\right\}.
\end{equation*}
\end{Def}
The quantity $\|\cdot\|_{\mathcal{A}_q^\alpha\left(\mathcal{D},X\right)}$ forms a (quasi-)norm on $\mathcal{A}_q^\alpha\left(\mathcal{D},X\right)$, and from \eqref{eq:LorentzEmbeddings} we immediately obtain the continuous embeddings
\begin{equation*}
\mathcal{A}_{q_1}^{\alpha_1}\left(\mathcal{D},X\right)\hookrightarrow \mathcal{A}_{q_2}^{\alpha_2}\left(\mathcal{D},X\right)\quad \text{if}\quad \alpha_1>\alpha_2\quad \text{or if}\quad \alpha_1=\alpha_2\text{ and }q_1\leq q_2.
\end{equation*}
We note the $f\in \mathcal{A}_q^\alpha\left(\mathcal{D},X\right)$ implies the decay in \eqref{eq:JacksonInequality} as desired. Following the approach in \cite{DeVore1996}, we define smoothness spaces as follows. 
\begin{Def}
Given $\tau\in (0,\infty), q\in (0,\infty],M>0$, we let
\begin{equation*}
\mathcal{K}_q^\tau(\mathcal{D},X,M):=\text{clos}_X\left\{\sum_{k\in \Delta}c_kg_k\in X \Big|\Delta\subset \bb{N},\abs{\Delta} <\infty,\norm{\{c_k\}_{k\in \Delta}}_{\ell_q^\tau}\leq M\right\}.
\end{equation*}
We then define $\mathcal{K}_q^\tau(\mathcal{D},X):=\cup_{M>0}\mathcal{K}_q^\tau(\mathcal{D},X,M)$ with
\begin{equation*}
\abs{f}_{\mathcal{K}_q^\tau(\mathcal{D},X)}:=\inf\left\{M>0~\Big|~f\in \mathcal{K}_q^\tau(\mathcal{D},X,M)\right\}.
\end{equation*}
\end{Def}
It can be shown that $|\cdot|_{\mathcal{K}_q^\tau(\mathcal{D},X)}$ is a (semi-quasi-)norm on $\mathcal{K}_q^\tau(\mathcal{D},X)$ and a (quasi-)norm if $\tau\in(0,1)$.
\begin{Rem}\label{Rem:GeneralSmoothnessSpace}
We note that for general $\mathcal{D}$ and $X$, $f\in \mathcal{K}_q^\tau(\mathcal{D},X)$ does not imply the existence of $\{c_k\}_{k\in \bb{N}}\in \ell_q^\tau$ with $f=\sum_{k\in \bb{N}}c_kg_k$. All realizations of $\mathcal{K}_q^\tau(\mathcal{D},X)$ considered in this article will, however, guarantee the existence of such reconstruction coefficients (cf. \Proref{Pro:hilbertianprop} below).
\end{Rem}
\begin{Exa}
Let $\mathcal{B}=\{g_k\}_{k\in \bb{N}}$ be an orthonormal basis for a Hilbert space $\mathcal{H}$. Given $\alpha\in (0,\infty), q\in (0,\infty]$ and $0<\tau=(\alpha+1/2)^{-1}<2$, we have the characterization
\begin{equation*}
\mathcal{A}_q^\alpha\left(\mathcal{B},\mathcal{H}\right)=\mathcal{K}_q^\tau\left(\mathcal{B},\mathcal{H}\right)=\left\{f\in \mathcal{H}~\Big|~\abs{f}_{\mathcal{K}_q^\tau(\mathcal{D},X)}=\norm{\{\scalarp{f,g_k}\}_{k\in \bb{N}}}_{\ell_q^\tau}<\infty\right\},
\end{equation*}
where the first equality is with equivalent (quasi-)norms \cite{Stechkin1955,DeVore1996}. 
\end{Exa}
Denoting the space of finite sequences on $\bb{N}$ by $\ell^0$, we define the reconstruction operator $R:\ell^0\rightarrow X$ by
\begin{equation}\label{eq:reconoperator}
R:\{c_k\}_{k\in \Delta}\rightarrow \sum_{k\in \Delta}c_kg_k,\quad \{c_k\}_{k\in \Delta}\in \ell^0.
\end{equation}
Recalling that $\|g_k\|_X=1$, for all $k\in \bb{N}$, we can extend this operator to a bounded operator from $\ell^1$ to $X$ since
\begin{equation}\label{eq:reconoperatorell1}
\norm{R\{c_k\}_{k\in \Delta}}_X\leq \sum_{k\in \Delta}\abs{c_k}\norm{g_k}_X=\norm{\{c_k\}_{k\in \Delta}}_{\ell^1},\quad \forall \{c_k\}_{k\in \Delta}\in \ell^0.
\end{equation}
Following the approach in \cite{Gribonval2004} we introduce the following class of dictionaries.
\begin{Def}\label{Def:hilbertdic}
Let $\mathcal{D}=\{g_k\}_{k\in \bb{N}}$ be a dictionary in a Banach space $X$. Given $\tau \in (0,\infty), q\in (0,\infty]$, we say that $\mathcal{D}$ is $\ell_q^\tau-$hilbertian if the reconstruction operator $R$ given in \eqref{eq:reconoperator} is bounded from $\ell_q^\tau$ to $X$.
\end{Def}
It follows from \eqref{eq:reconoperatorell1} and \eqref{eq:LorentzEmbeddings} that every $\mathcal{D}$ is $\ell_q^\tau-$hilbertian if $\tau<1$. According to \cite[Proposition 3]{Gribonval2004} we have the following characterization.\newpage
\begin{Pro}\label{Pro:hilbertianprop}
Assume $\mathcal{D}$ is $\ell_1^p-$hilbertian with $p\in (1,\infty)$. Let $\tau\in (0,p)$ and $q\in (0,\infty]$. For all $f\in \mathcal{K}_q^\tau(\mathcal{D},X)$, there exists some $\bm{c}:=\bm{c}_{\tau,q}(f)\in \ell_q^\tau$ with $f=R\bm{c}$ and $\|\bm{c}\|_{\ell_q^\tau}=|f|_{\mathcal{K}_q^\tau(\mathcal{D},X)}$. If $1< q\leq \tau<\infty$, then $\bm{c}$ is unique. Consequently
\begin{equation*}
\abs{f}_{\mathcal{K}_q^\tau(\mathcal{D},X)}=\min_{\bm{c}\in \ell_q^\tau, f=R\bm{c}}\norm{\bm{c}}_{\ell_q^\tau},
\end{equation*}
and
\begin{equation*}
\mathcal{K}_q^\tau(\mathcal{D},X)=\left\{\sum_{k\in \bb{N}}c_kg_k\in X ~\Bigg|~\norm{\{c_k\}_{k\in \bb{N}}}_{\ell_q^\tau}< \infty\right\}
\end{equation*}
is a (quasi-)Banach space with $\mathcal{K}_q^\tau(\mathcal{D},X)\hookrightarrow X$.
\end{Pro}
In the next section we describe the proposed algorithm using the framework presented in this section.
\section{Weighted thresholding}\label{Sec:3}
Let $\mathcal{D}:=\{g_k\}_{k\in \bb{N}}$ be a complete dictionary for a Banach space $X$ and let $R$ be the reconstruction operator defined in \eqref{eq:reconoperator}. Given $\{c_k\}_{k\in \bb{N}}\subset \bb{C}$, we let $\pi:\bb{N}\rightarrow \bb{N}$ denote a bijective mapping such that $\{|c_{\pi(k)}|\}_{k\in \bb{N}}$ is non-increasing, i.e, $\{|c_{\pi(k)}|\}_{k\in \bb{N}}=\{c^*_k\}_{k\in \bb{N}}$. For $f\in X$, $\{c_k\}_{k\in \bb{N}}\subset \bb{C}$, and $m\in \bb{N}$, a standard way of constructing an $m$-term approximant to $f$ from $\mathcal{D}$ is by thresholding
\begin{equation}\label{eq:standardgreedyapproximant}
f_m:=f_m(\pi,\{c_k\}_{k\in \bb{N}},\mathcal{D}):=R\{c_{\pi(k)}\}_{k=1}^m=\sum_{k=1}^mc_{\pi(k)}g_{\pi(k)}.
\end{equation}
For all practical purpose we choose $\{c_k\}_{k\in \bb{N}}\subset \bb{C}$ as reconstruction coefficients for $f$ such that $f=R\{c_k\}_{k\in \bb{N}}$ (cf. \Remref{Rem:GeneralSmoothnessSpace}). With this choice, the approximants $\{f_m\}_{m\in \bb{N}}$ converge to $f$ as $m\rightarrow \infty$. 

The thresholding approach in \eqref{eq:standardgreedyapproximant} chooses the $m$ elements from $\mathcal{D}$ corresponding to the $m$ largest of the coefficients $\{c_k\}_{k\in \bb{N}}$. As mentioned in the introduction, many real world signals have an inherent structure between the expansion coefficients, which should be accounted for in the approximation procedure. Therefore, we would like to construct an algorithm which preserves local coherence, such that a small coefficient $c_1$ might be preserved, in exchange for a larger (isolated) coefficient $c_2$, if $c_1$ belongs to a neighbourhood with many large coefficients. This leads us to consider banded Toeplitz matricies. 

Let $\Lambda$ denote a banded non-negative Toeplitz matrix with bandwidth $\Omega\in \bb{N}$, i.e., the non-zero entries $\Lambda_{(i,j)}$ satisfy $|i-j|\leq\Omega$. We assume the scalar on the diagonal of $\Lambda$ is positive and we define $\{c_k^\Lambda\}_{k\in \bb{N}}:=\Lambda(\{\abs{c_k}\}_{k\in \bb{N}})$ for $\{c_k\}_{k\in \bb{N}}\subset \bb{C}$.
\begin{Exa}
With $\Omega=2$ we obtain
\begin{equation*}
\{c_k^\Lambda\}_{k\in \bb{N}}=
\begin{bmatrix}
\lambda_0 & \lambda_1 & \lambda_2 & 0 & 0 & 0 & \cdots\\
\lambda_{-1} & \lambda_0 & \lambda_1 & \lambda_2 & 0 & 0 & \cdots\\
\lambda_{-2} & \lambda_{-1} & \lambda_0 & \lambda_1 & \lambda_2 & 0& \cdots\\
0 & \lambda_{-2} & \lambda_{-1} & \lambda_0 & \lambda_1 & \lambda_2 & \cdots\\
\vdots & \vdots & \vdots & \vdots & \vdots & \vdots & \ddots
\end{bmatrix}
\begin{bmatrix}
\abs{c_1}\\
\abs{c_2}\\
\abs{c_3}\\
\abs{c_4}\\
\vdots
\end{bmatrix},
\end{equation*}
with $\lambda_0>0$ and $\lambda_l\geq 0$ for all $l\in\{-2,-1,1,2\}$. We note that
\begin{equation*}
c_k^\Lambda=\lambda_{-2}\abs{c_{k-2}}+\cdots+\lambda_{0}\abs{c_{k}}+\cdots+\lambda_{2}\abs{c_{k+2}}=\sum_{j=k-\Omega}^{k+\Omega}\lambda_{j-k}\abs{c_{j}},
\end{equation*}
for all $k\in \bb{N}$. 
\end{Exa}
Given $\{c_k\}_{k\in \bb{N}}\subset \bb{C}$, we let $\pi_\Lambda:\bb{N}\rightarrow \bb{N}$ denote a bijective mapping such that the sequence $\{c_{\pi_\Lambda(k)}^\Lambda\}_{k\in \bb{N}}$ is non-increasing. If $\Lambda$ is the identity map, we write $\pi$ instead of $\pi_\Lambda$ to be consistent with the notation of \eqref{eq:standardgreedyapproximant}. We shall use the following technical result in the proof of \Theref{The:mainresult}.
\begin{Lem}\label{Lem:sequencehelper}
Let $\Lambda$ be a banded non-negative Toeplitz matrix with bandwidth $\Omega\in \bb{N}$ and let $p \in (0,\infty), q\in (0,\infty]$. There exists a constant $C>0$, such that for any non-increasing sequence $\{c_k\}_{k\in \bb{N}}\in \ell_q^p$ we have the estimate
\begin{equation*}
\norm{\{c_{\pi_\Lambda(k)}\}_{k=m+1}^\infty}_{\ell_q^p}\leq C\norm{\{c_k\}_{k=m+1-\Omega}^\infty}_{\ell_q^p},\quad \forall m\geq \Omega.
\end{equation*}
\end{Lem}
\begin{proof}
Fix $m\in \bb{N}$ and let $\bm{a}:=\{|c_{\pi_\Lambda(k)}|\}_{k=m+1}^\infty$ and $\bm{b}:=\{|c_{k}|\}_{k=m+1}^\infty$. If $\bm{a}$ contains the same coefficients as $\bm{b}$, then there is nothing to prove. If this is not the case, then there exists $k'\leq m$, with $|c_{k'}|\in \bm{a}$, and $k''\geq m+1$, with $|c_{k''}|\notin \bm{a}$, satisfying
\begin{equation*}
\sum_{j=k'-\Omega}^{k'+\Omega}\lambda_{j-k'}\abs{c_{j}}\leq \sum_{j=k''-\Omega}^{k''+\Omega}\lambda_{j-k''}\abs{c_j}\Rightarrow\abs{c_{k'}}\leq \frac{1}{\lambda_0}\sum_{j=k''-\Omega}^{k''+\Omega}\lambda_{j-k''}\abs{c_j}.
\end{equation*}
Denoting the maximum value of the $\lambda_l$'s by $\lambda_{\text{max}}$ and the maximum value of the $\abs{c_j}$'s, for $j\in [k''-\Omega,k''+\Omega]$, by $\abs{c_{\tilde{k}}}$, we thus get
\begin{equation*}
\abs{c_{k'}}\leq \frac{\lambda_{\text{max}}}{\lambda_0}(2\Omega+1)\abs{c_{\tilde{k}}}.
\end{equation*}
Since $\tilde{k}\in [k''-\Omega,k''+\Omega]$, and $k''\geq m+1$, then $\tilde{k}\in [m+1-\Omega,\infty)$. We conclude that $c_{\tilde{k}}\in \{c_k\}_{k=m+1-\Omega}^\infty$. The lemma follows directly from this observation.\qed
\end{proof}
Given $f\in X$, a set of reconstruction coefficients $\{c_k\}_{k\in \bb{N}}\subset \bb{C}$, and $m\in \bb{N}$, we generalize the notation of \eqref{eq:standardgreedyapproximant} and construct an $m$-term approximant to $f$ by
\begin{equation}\label{eq:weightthresholdingapproximant}
f^\Lambda_m:=f^\Lambda_m(\pi_\Lambda,\{c_k\}_{k\in \bb{N}},\mathcal{D}):=R\{c_{\pi_\Lambda(k)}\}_{k=1}^m=\sum_{k=1}^mc_{\pi_\Lambda(k)}g_{\pi_\Lambda(k)}.
\end{equation}
If $\Lambda$ is the identity map, we just obtain the approximant $f_m$ given in \eqref{eq:standardgreedyapproximant}. If not, we obtain an approximant which chooses the elements of $\{g_k\}_{k\in \bb{N}}$ corresponding to the indices of the $m$ largest of the weighted coefficients $\{c_k^\Lambda\}_{k\in \bb{N}}$. We generalize the notation of \cite{Gribonval2004} and define weighted thresholding approximation spaces as follows.	
\begin{Def}
Let $\Lambda$ be a banded non-negative Toeplitz matrix with bandwidth $\Omega\in \bb{N}$. Given $\alpha\in (0,\infty), q\in (0,\infty]$, we define
\begin{equation*}
\mathcal{T}_q^\alpha\left(\mathcal{D},X,\Lambda\right):=\left\{f\in X~\Big|~\norm{f}_{\mathcal{T}_q^\alpha\left(\mathcal{D},X,\Lambda\right)}:=\abs{f}_{\mathcal{T}_q^\alpha\left(\mathcal{D},X,\Lambda\right)}+\norm{f}_X<\infty\right\},
\end{equation*}
with
\begin{equation*}
\abs{f}_{\mathcal{T}_q^\alpha\left(\mathcal{D},X,\Lambda\right)}:=\left\{
     \begin{array}{ll}
       \inf_{\pi_\Lambda, \{c_k\}_{k\in \bb{N}}}\left(\sum_{m=1}^\infty \left[m^\alpha\norm{f-f^\Lambda_m}_X\right]^q \frac{1}{m}\right)^{1/q},  & 0<q<\infty\\
        \inf_{\pi_\Lambda, \{c_k\}_{k\in \bb{N}}}\left(\sup_{m\geq 1}m^\alpha\norm{f-f^\Lambda_m}_X\right),  & q=\infty
     \end{array}
   \right.
\end{equation*}
When $\Lambda$ is the identity mapping, we write $\mathcal{T}_q^\alpha\left(\mathcal{D},X\right)$ instead of $\mathcal{T}_q^\alpha\left(\mathcal{D},X,\Lambda\right)$.
\end{Def}
\begin{Rem}
We note that the expression for $|f|_{\mathcal{T}_q^\alpha\left(\mathcal{D},X,\Lambda\right)}$ cannot be written using the Lorentz norm, since the sequence $\{\|f-f^\Lambda_m\|_X\}_{m\in \bb{N}}$ might not be non-increasing.
\end{Rem}
In order to prove \Theref{The:mainresult} we need to impose further assuptions on the dictionary $\mathcal{D}$. We call $\mathcal{D}$ an atomic decomposition (AD) \cite{Feichtinger89I,Feichtinger89II} for $X$, with respect to $\ell_q^\tau$, if there exists a sequence $\{\tilde{g}_k\}_{k\in \bb{N}}$, in the dual space $X'$, such that
\begin{enumerate}
\item There exist $0<C'\leq C''<\infty$ with
\begin{equation*}
C'\norm{\{\scalarp{f,\tilde{g}_k}\}_{k\in \bb{N}}}_{\ell_q^\tau}\leq \norm{f}_X\leq C'' \norm{\{\scalarp{f,\tilde{g}_k}\}_{k\in \bb{N}}}_{\ell_q^\tau},\quad \forall f\in X.
\end{equation*}

\item The reconstruction operator $R$ given in \eqref{eq:reconoperator} is bounded from $\ell_q^\tau$ onto $X$ and we have the expansions
\begin{equation*}
R(\{\langle f,\tilde{g}_k\rangle\}_{k\in \bb{N}})=\sum_{k\in \bb{N}}\scalarp{f,\tilde{g}_k}g_k=f,\quad \forall f\in X.
\end{equation*}
\end{enumerate}
\begin{Exa}
Standard examples of ADs are Gabor frames for modulation spaces \cite{Grochenig2001} and wavelets for Besov spaces \cite{DeVore1992}. However, many other examples have been constructed in the general framework of decomposition spaces \cite{Feichtinger85,Feichtinger87,Borup07}. For instance curvelets, shearlets and nonstationary Gabor frames (or generalized shift-invariant systems) \cite{Borup07,Voigtlaender2017,Ottosen2017}.
\end{Exa}
We can now present the main result of this article. 
\begin{The}\label{The:mainresult}
Let $\Lambda$ be a banded non-negative Toeplitz matrix with bandwidth $\Omega\in \bb{N}$ and let $p\in (1,\infty)$, $\tau\in (0,p)$, $q\in (0,\infty]$. If $\mathcal{D}=\{g_k\}_{k\in \bb{N}}$ is $\ell_1^p-$hilbertian and forms an AD for $X$, with respect to $\ell_q^\tau$, then
\begin{equation*}
\mathcal{K}_q^\tau(\mathcal{D},X)\hookrightarrow \mathcal{T}_q^\alpha\left(\mathcal{D},X,\Lambda\right)\hookrightarrow \mathcal{A}_q^\alpha\left(\mathcal{D},X\right),
\end{equation*}
with $\alpha=1/\tau-1/p>0$.
\end{The}
\begin{proof}
The embedding $\mathcal{T}_q^\alpha\left(\mathcal{D},X,\Lambda\right)\hookrightarrow \mathcal{A}_q^\alpha\left(\mathcal{D},X\right)$ follows directly from the definitions of these spaces. Let us now show that $\mathcal{K}_q^\tau(\mathcal{D},X)\hookrightarrow \mathcal{T}_q^\alpha\left(\mathcal{D},X,\Lambda\right)$. Since the Lorentz spaces are rearrangement invariant, we may assume the canonical coefficients $\bm{d}:=\{d_k\}_{k\in \bb{N}}:=\{\langle f,\tilde{g}_k\rangle\}_{k\in \bb{N}}$ form a non-increasing sequence. Given $f\in \mathcal{K}_q^\tau(\mathcal{D},X)$, the $\ell_1^p-$hilbertian property of $\mathcal{D}$ implies
\begin{equation*}
\norm{f}_X=\norm{R\bm{d}}_X\leq C_1\norm{\bm{d}}_{\ell_1^p}.
\end{equation*}
Defining $f^\Lambda_m(\pi_\Lambda,\bm{d},\mathcal{D})=R\{d_{\pi_\Lambda(k)}\}_{k=1}^m$ as in \eqref{eq:weightthresholdingapproximant}, we thus get
\begin{align}\label{eq:normestimateweightthres}
\norm{f}_{\mathcal{T}_q^\alpha\left(\mathcal{D},X,\Lambda\right)}&=\inf_{\pi_\Lambda, \{c_k\}_{k\in \bb{N}}}\left(\sum_{m=1}^\infty \left[m^\alpha\norm{f-f^\Lambda_m}_X\right]^q \frac{1}{m}\right)^{1/q}+\norm{f}_X\notag\\
&\leq \left(\sum_{m=1}^\infty \left[m^\alpha\norm{f-f^\Lambda_m(\pi_\Lambda,\bm{d},\mathcal{D})}_X\right]^q \frac{1}{m}\right)^{1/q}+C_1\norm{\bm{d}}_{\ell_1^p}.
\end{align}
Now, since
\begin{align*}
\norm{f-f^\Lambda_m(\pi_\Lambda,\bm{d},\mathcal{D})}_X&=\norm{R\{d_{\pi_\Lambda(k)}\}_{k=m+1}^\infty}_X\\
&\leq C_1\norm{\{d_{\pi_\Lambda(k)}\}_{k=m+1}^\infty}_{\ell_1^p}\leq C_1\norm{\bm{d}}_{\ell_1^p},
\end{align*}
we get the following estimate for the first $\Omega$ terms in \eqref{eq:normestimateweightthres}
\begin{equation}\label{eq:helper1}
\sum_{m=1}^{\Omega}\left[m^\alpha\norm{f-f^\Lambda_m(\pi_\Lambda,\bm{c},\mathcal{D})}_X\right]^q \frac{1}{m}\leq C_2\norm{\bm{d}}_{\ell_1^p}^q.
\end{equation}
For $m\geq \Omega+1$, \Lemref{Lem:sequencehelper} implies 
\begin{align*}
\norm{f-f_m^\Lambda(\pi_\Lambda,\bm{d},\mathcal{D})}_X&\leq C_1\norm{\{d_{\pi_\Lambda(k)}\}_{k=m+1}^\infty}_{\ell_1^p}\leq C_3\norm{\{d_{k}\}_{k=m+1-\Omega}^\infty}_{\ell_1^p}\\
&=C_3\norm{\bm{d}-\{d_{k}\}_{k=1}^{m-\Omega}}_{\ell_1^p}=C_3\sigma_{m-\Omega}(\bm{d},\mathcal{B})_{\ell_1^p},
\end{align*}
with $\mathcal{B}$ denoting the canonical basis of $\ell_1^p$. Hence,
\begin{align}\label{eq:helper2}
\sum_{m=\Omega+1}^\infty \left[m^\alpha\norm{f-f^\Lambda_m(\pi_\Lambda,\bm{d},\mathcal{D})}_X\right]^q \frac{1}{m}&\leq C_3\sum_{m=\Omega+1}^\infty \frac{\left[m^\alpha\sigma_{m-\Omega}(\bm{d},\mathcal{B})_{\ell_1^p}\right]^q}{m}\notag\\
&=C_3\sum_{m=1}^\infty \frac{\left[(m+\Omega)^\alpha\sigma_{m}(\bm{d},\mathcal{B})_{\ell_1^p}\right]^q}{m+\Omega}\notag\\
&\leq C_4\norm{\{\sigma_m(\bm{d},\mathcal{B})_{\ell_1^p}\}_{m\in \bb{N}}}_{\ell_q^{1/\alpha}}^q.
\end{align}
Combining \eqref{eq:helper1} and \eqref{eq:helper2}, then \eqref{eq:normestimateweightthres} yields
\begin{equation*}
\norm{f}_{\mathcal{T}_q^\alpha\left(\mathcal{D},X,\Lambda\right)}\leq C_5\left(\norm{\{\sigma_m(\bm{d},\mathcal{B})_{\ell_1^p}\}_{m\in \bb{N}}}_{\ell_q^{1/\alpha}}+\norm{\bm{d}}_{\ell_1^p}\right)=C_5\norm{\bm{d}}_{\mathcal{A}_q^\alpha\left(\mathcal{B},\ell_1^p\right)}.
\end{equation*}
Applying \cite[Theorem 3.1]{Nielsen2001} we thus get
\begin{equation}\label{eq:helper3}
\norm{f}_{\mathcal{T}_q^\alpha\left(\mathcal{D},X,\Lambda\right)}\leq C_6\abs{\bm{d}}_{\mathcal{K}_q^\tau(\mathcal{B},\ell_1^p)}=C_6\norm{\bm{d}}_{\ell_q^\tau}\leq C_7\norm{f}_{X}.
\end{equation}
Finally, since $\mathcal{D}$ is $\ell_1^p-$hilbertian, \Proref{Pro:hilbertianprop} states that we can find $\bm{c}\in \ell_q^\tau$ with $f=R\bm{c}$ and $\|\bm{c}\|_{\ell_q^\tau}=|f|_{\mathcal{K}_q^\tau(\mathcal{D},X)}$, such that \eqref{eq:helper3} yields
\begin{equation*}
\norm{f}_{\mathcal{T}_q^\alpha\left(\mathcal{D},X,\Lambda\right)}\leq C_7\norm{R\bm{c}}_{X}\leq C_8\norm{\bm{c}}_{\ell_q^\tau}=C_8\abs{f}_{\mathcal{K}_q^\tau(\mathcal{D},X)}.
\end{equation*}
This completes the proof. \qed
\end{proof}
\begin{Rem}
The assumption in \Theref{The:mainresult} of $\mathcal{D}$ being $\ell_1^p-$hilbertian directly implies the boundedness of $R:\ell_q^\tau\rightarrow X$ in the definition of an AD since $\tau\in (0,p)$. It should also be noted that if $\Lambda$ is the identity operator, then \Theref{The:mainresult} holds without the assumption of an AD --- this was proven in \cite[Theorem 6]{Gribonval2004}. However, in contrast to the proof presented here, there is no constructive way of obtaining the sparse expansion coefficients in the proof of \cite[Theorem 6]{Gribonval2004}.
\end{Rem}
The Jackson embedding in \Theref{The:mainresult} is strong in the sense that it provides an associated constructive algorithm which obtains the approximation rate. Given $f\in X$, the algorithm goes as follows:
\begin{enumerate}
\item Calculate the canonical coefficients $\{d_k\}_{k\in \bb{N}}=\{\langle f,\tilde{g}_k\rangle\}_{k\in \bb{N}}$.

\item Construct the weighted coefficients $\{d^\Lambda_k\}_{k\in \bb{N}}$ according to $\Lambda$ by
\begin{equation*}
d_k^\Lambda=\sum_{j=k-\Omega}^{k+\Omega}\lambda_{j-k}\abs{d_{j}},\quad k\in \bb{N}.
\end{equation*}

\item Choose $\pi_\Lambda:\bb{N}\rightarrow \bb{N}$ such that $\{d_{\pi_\Lambda(k)}^\Lambda\}_{k\in \bb{N}}$ is non-increasing.

\item Construct an $m$-term approximation to $f$ by
\begin{equation*}
f^\Lambda_m(\pi_\Lambda,\{d_k\}_{k\in \bb{N}},\mathcal{D})=\sum_{k=1}^md_{\pi_\Lambda(k)}g_{\pi_\Lambda(k)}.
\end{equation*}
\end{enumerate}
With this construction, \Theref{The:mainresult} states that
\begin{equation*}
\norm{f-f^\Lambda_m(\pi_\Lambda,\{d_k\}_{k\in \bb{N}},\mathcal{D})}_{X}=\mathcal{O}(m^{-\alpha}),\quad \forall f\in \mathcal{K}_q^\tau(\mathcal{D},X),
\end{equation*}
with $\alpha=1/\tau-1/p>0$. In the next section we consider the particular case where the dictionary is a Gabor frame and the smoothness space is a modulation space.

\section{Modulation spaces and Gabor frames}\label{Sec:4}
In this section we choose $X$ as the modulation space $M^p:=M^p(\bb{R}^d)$, with $1\leq p<\infty$, consisting of all 
$f\in \mathcal{S}'(\bb{R}^d)$ satisfying
\begin{equation*}
\norm{f}_{M^p}:=\left(\int_{\bb{R}^d}\int_{\bb{R}^d}\abs{V_{\gamma}f(x,y)}^pdxdy\right)^{1/p}<\infty.
\end{equation*}
Here, $V_{\gamma}f(x,y)$ denotes the short-time Fourier transform of $f$ with respect to a window function $\gamma\in \mathcal{S}(\bb{R}^d)\setminus\{0\}$ . It can be shown that $M^p$ is independent of the particular choice of window function and different choices yield equivalent norms \cite{Grochenig2001}. Given $g\in \mathcal{S}(\bb{R}^d)\setminus\{0\}$ and lattice parameters $a,b\in (0,\infty)$, we consider the Gabor system 
\begin{equation*}
g_{m,n}(t):=g(t-na)e^{2\pi imb\cdot t},\quad t\in \bb{R}^d.
\end{equation*}
Assuming $\{g_{m,n}\}_{m,n}$ forms a frame for $L^2(\bb{R}^d)$ (see \cite{Christensen2016} for details), there exists a dual frame $\{\tilde{g}_{m,n}\}_{m,n}$ such that $\{g_{m,n}\}_{m,n\in \bb{Z}^d}$ is an AD for $M^p$ with respect to $\ell^p$ for all $1\leq p<\infty$\cite{Grochenig2001}. We have the following version of \cite[Proposition 3]{Grochenig2000}.
\begin{Pro}\label{eq:Promodulationcharac}
Let $\Lambda$ be a banded non-negative Toeplitz matrix with bandwidth $\Omega\in \bb{N}$. Let $1\leq \tau<p<\infty$ and $g\in \mathcal{S}(\bb{R}^d)\setminus\{0\}$. If $\mathcal{D}=\{g_{m,n}\}_{m,n\in \bb{Z}^d}$ forms a Gabor frame for $L^2(\bb{R}^d)$ and $\|g_{m,n}\|_{M^p}=1$ for all $m,n\in \bb{Z}^d$, then 
\begin{equation*}
M^\tau=\mathcal{K}_\tau^\tau(\mathcal{D},M^p)\hookrightarrow \mathcal{T}^\alpha_\tau(\Lambda,\mathcal{D},M^p)\hookrightarrow \mathcal{A}_\tau^\alpha\left(\mathcal{D},M^p\right),\quad \alpha=1/\tau-1/p,
\end{equation*}
where the first equality is with equivalent norms.
\end{Pro}
\begin{proof}
We first note that $\mathcal{D}$ simultaneously forms an AD for both $M^p$ and $M^\tau$. Since $\mathcal{D}$ constitutes an AD for $M^p$, and $p>1$, we get
\begin{equation*}
\norm{R\bm{c}}_{M^p}\leq C_1\norm{\bm{c}}_{\ell^p}\leq C_2\norm{\bm{c}}_{\ell_1^p},\quad \forall \bm{c}\in \ell_1^p,
\end{equation*}
which shows that $\mathcal{D}$ is a $\ell_1^p-$hilbertian dictionary for $M^p$. Hence, the embeddings in \Proref{eq:Promodulationcharac} follows from \Theref{The:mainresult}. Let us now prove that $M^\tau=\mathcal{K}_\tau^\tau(\mathcal{D},M^p)$ with equivalent norms. According to \Proref{Pro:hilbertianprop} then
\begin{align*}
\mathcal{K}_\tau^\tau(\mathcal{D},M^p)&=\left\{\sum_{m\in \bb{Z}^d}\sum_{n\in \bb{Z}^d}c_{m,n}g_{m,n}\in M^p ~\Big|~\norm{\{c_{m,n}\}_{m,n\in \bb{Z}^d}}_{\ell^\tau}< \infty\right\}\quad \text{with}\\
\abs{f}_{\mathcal{K}_\tau^\tau(\mathcal{D},M^p)}&=\min_{\bm{c}\in \ell^\tau, f=R\bm{c}}\norm{\bm{c}}_{\ell^\tau},\quad f\in \mathcal{K}_\tau^\tau(\mathcal{D},M^p).
\end{align*}
Since $\mathcal{D}$ constitutes an AD for $M^\tau$, then for $f\in M^\tau$ we have
\begin{equation*}
f=\sum_{m\in \bb{Z}^d}\sum_{n\in \bb{Z}^d}\scalarp{f,\tilde{g}_{m,n}}g_{m,n},\quad\text{with}\quad \norm{\left\{\scalarp{f,\tilde{g}_{m,n}}\right\}_{m,n\in \bb{Z}^d}}_{\ell^\tau}\leq C \norm{f}_{M^\tau}.
\end{equation*}
As $\tau<p$ then $M^\tau\subseteq M^p$ (cf. \cite{Grochenig2001}), which shows that $f\in \mathcal{K}_\tau^\tau(\mathcal{D},M^p)$ and $|f|_{\mathcal{K}_\tau^\tau(\mathcal{D},M^p)}\leq C \|f\|_{M^\tau}$. The converse embedding follows from 
\begin{equation*}
\norm{f}_{M^\tau}=\norm{R\bm{c}}_{M^\tau}\leq C'\norm{\bm{c}}_{\ell^\tau}=C'|f|_{\mathcal{K}_\tau^\tau(\mathcal{D},M^p)},\quad f\in \mathcal{K}_\tau^\tau(\mathcal{D},M^p).
\end{equation*}
This completes the proof. \qed
\end{proof}
In the next section we apply the proposed method for denoising music signals with elements from a Gabor dictionary. For an introduction to Gabor theory in the discrete settings, we refer the reader to \cite{Sondergaard2007,Strohmer1998}.

\section{Numerical experiments}\label{Sec:5}
For the implementation we use MATLAB 2017B and apply the routines from the following two toolboxes: The "Large time-frequency analysis toolbox" (LTFAT) version 2.2.0 \cite{ltfatnote030} avaliable from 
\begin{center}
\url{http://ltfat.sourceforge.net/}
\end{center}
and the StrucAudioToolbox \cite{Siedenburg2011} avaliable from 
\begin{center}
\url{http://homepage.univie.ac.at/monika.doerfler/StrucAudio.html}.
\end{center}
All Gabor transforms are constructed using 1024 frequency channels, a hop size of 256, and a Hanning window of length 1024 (this is the default settings in the StrucAudioToolbox). These settings lead to transforms of redundancy of four, meaning there are four times as many time-frequency coefficients as signal samples. The music signals we consider are part of the EBU-SQAM database \cite{EBUSQAM}, which consists of 70 test sounds sampled at 44 kHz. The database contains a large variety of different music sounds including single instruments, vocal, and orchestra. We measure the reconstrucing error of an algorithm by the the relative root mean square (RMS) error
\begin{equation*}
\text{RMS}(f,f_{rec}):=\frac{\norm{f-f_{rec}}_2}{\norm{f}_2}.
\end{equation*}
We begin by analyzing the first $524288$ samples of signal $8$ in the EBU-SQAM database, which consists of an increasing melody of 10 tones played by a violin. A noisy version of the signal is constructed by adding white Guassian noise and the resulting spectrograms can be found in \Figref{fig:CleanNoisySpec}.
\begin{figure}[ht]
\center
\includegraphics[width=\textwidth]{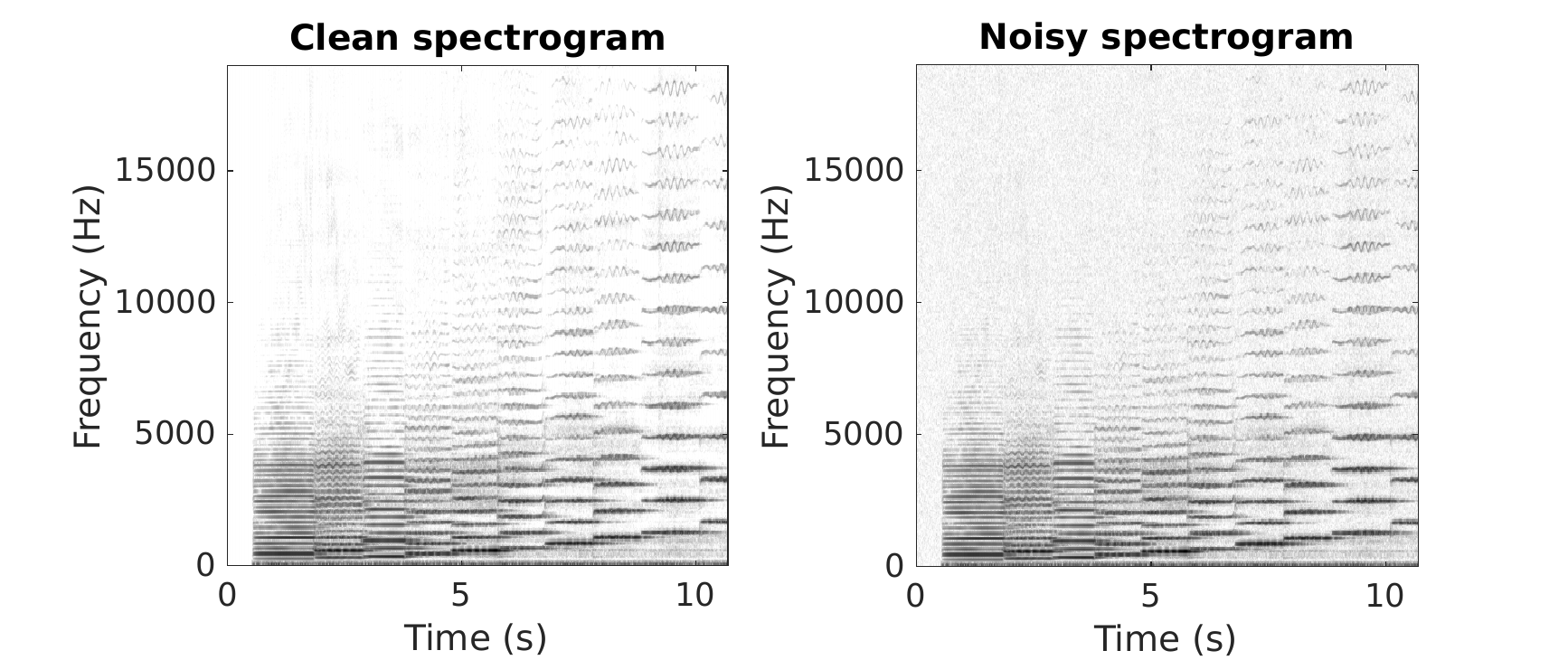}
\caption{Clean and noisy spectrograms of violin music.}
\label{fig:CleanNoisySpec}      
\end{figure}

For the task of denoising we first compare the greedy thresholding approach from nonlinear approximation theory (cf. \eqref{eq:standardgreedyapproximant}) with the Windowed-Group-Lasso (WGL) from social sparsity. For the WGL we use the default settings of the StrucAudioToolbox, which applies a horizontal assymetric neighbourhood for the shrinkage operator, see \cite{Siedenburg2011} for further details. The WGL constructs a denoised version of the spectrogram using only $74739$ non-zero coefficients (out of a total of $10506244$ coefficients). Using the same number of non-zero coefficients for the greedy thresholding approach we obtain the spectrograms shown in \Figref{fig:BestMTermWGL}.
\begin{figure}[ht]
\center
\includegraphics[width=\textwidth]{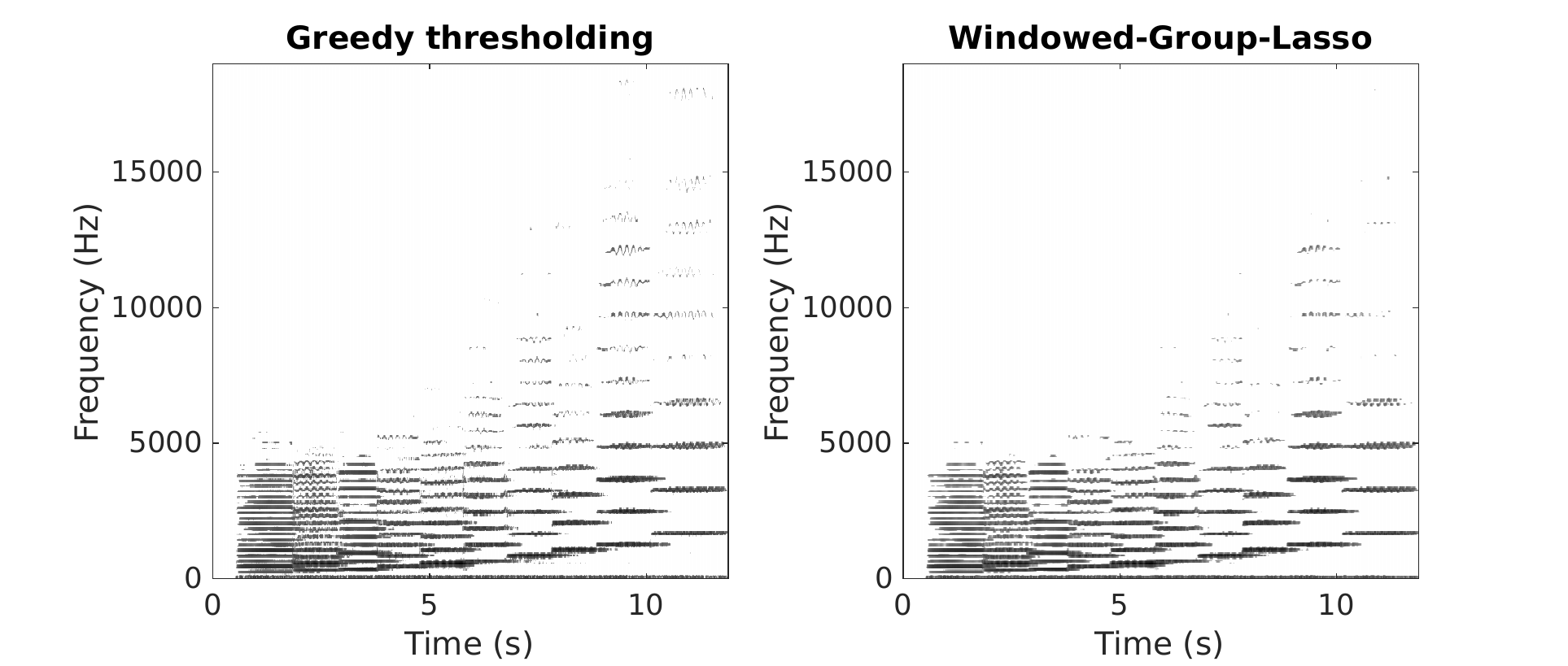}
\caption{Greedy thresholding and the WGL with $74739$ non-zero coefficients.}
\label{fig:BestMTermWGL}      
\end{figure}

We note from \Figref{fig:BestMTermWGL} that the greedy thresholding approach includes more coefficients at higher frequencies than the WGL. On the other hand, the WGL includes more coeffcients at lower frequencies, resulting in a smoother resolution for the fundamental frequencies and the first harmonics. This illustrates the way the WGL is designed, namely that a large isolated coefficient may be discarded in exhange for a smaller coefficient with large neighbours. The RMS error is $\approx 0.084$ for the WGL and $\approx 0.034$ for the greedy thresholding algorithm. To visualize the performance of the proposed algorithm we choose a rather extreme (horizontal) weight with
\begin{equation}\label{eq:numweight}
c_{m,n}^\Lambda=\abs{c_{m,n-2}}+\abs{c_{m,n-1}}+\abs{c_{m,n}}+\abs{c_{m,n+1}}+\abs{c_{m,n+2}}.
\end{equation} 
We then choose the $74739$ coefficients with largest weighted magnitudes according to \eqref{eq:numweight}. In \Figref{fig:BestMTermProposed} we have compared the resulting spectrogram against the spectrogram obtained using the greedy thresholding approach.
\begin{figure}[ht]
\center
\includegraphics[width=\textwidth]{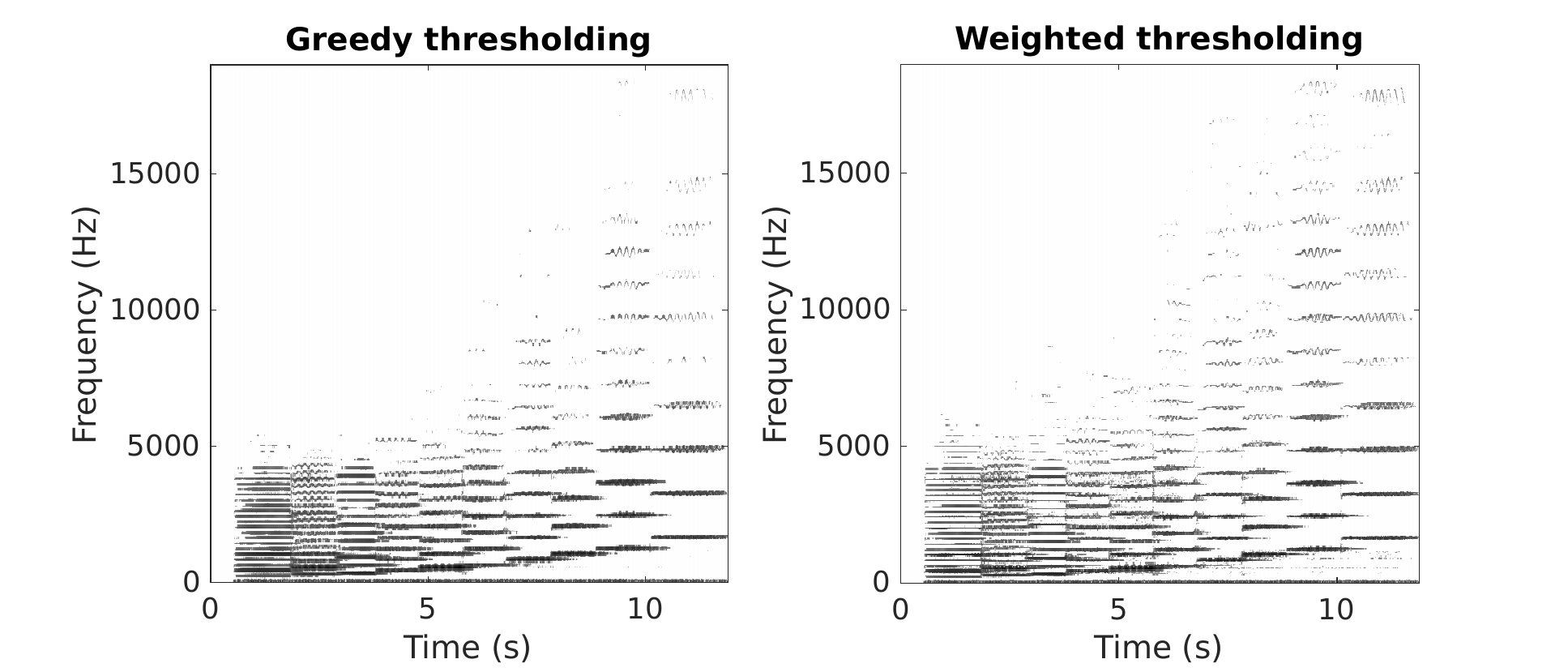}
\caption{Greedy thresholding and weighted thresholding with $74734$ non-zero coefficients.}
\label{fig:BestMTermProposed}      
\end{figure}

We can see from \Figref{fig:BestMTermProposed} that the weight in \eqref{eq:numweight} enforces the structures at higher frequencies even further than the greedy approach. This might be desirable for some applications as the timbre of the instrument is determined by the harmonics. The RMS error associated with the weighted thresholding approach is $\approx 0.074$, which is higher than for the greedy approach but lower than for the WGL. Applying a more moderate weight (for instance Weight 2 defined in \eqref{eq:threeweights} below) we obtain an RMS error of $\approx 0.031$, which is lower than for the greedy approach.

We now extend the experiment described above to the entire EBU-SQAM database. For the weighted thresholding we consider the following three weights
\begin{align}\label{eq:threeweights}
\text{Weight 1: }c_{m,n}^\Lambda&=\abs{c_{m,n}}+(\abs{c_{m-1,n}}+\abs{c_{m+1,n}})/2,\notag \\
\text{Weight 2: }c_{m,n}^\Lambda&=\abs{c_{m,n}}+(\abs{c_{m,n-1}}+\abs{c_{m,n+1}})/2,\\
\text{Weight 3: }c_{m,n}^\Lambda&=\abs{c_{m,n}}+(\abs{c_{m,n-1}}+\abs{c_{m-1,n}}+\abs{c_{m,n+1}}+\abs{c_{m+1,n}})/4.\notag
\end{align}

For each of the 70 test sounds in the EBU-SQAM database, we apply the WGL and calculate the associated RMS error and number of non-zero coefficients. Using the same number of non-zero coeffcients, we then apply the greedy thresholding approach and the three weighted thresholding algorithms defined in \eqref{eq:threeweights}. The resulting averaged values can be found in \Tabref{tab:AverageResults}.
\begin{table}[h!]
\centering
\caption{Average RMS errors over the EBU-SQAM database for the WGL, the greedy thresholding approach, and the three weighted thresholding algorithms defined in \eqref{eq:threeweights}.}
\label{tab:AverageResults}
\begin{tabular}{l c c c c c}
\hline
Algorithm: & WGL & Greedy & Weight 1 & Weight 2 & Weight 3 \\
\hline
Average RMS error.: & $0.1031$ & $0.0462$ & $0.0511$ & $0.0453$ & $0.0472$\\
\hline
\end{tabular}
\end{table}

The average number of coefficients was $68983$, which corresponds to $6.5\%$ of the total number of coefficients. We note that the RMS error associated with the WGL is rougly twice as large as for the various thresholding algorithms. We also note that the smallest RMS error is obtained by the weighted thresholding approach, which applies a horizontal weight (Weight 2). This is likely due to the horizontal structure of the harmonics as seen in \Figref{fig:CleanNoisySpec}.

Let us mention that it might be possible to reduce the error for the WGL by tuning the parameters instead of using the default settings (cf. \cite{Siedenburg2012} for a detailed analysis of the parameter settings for the WGL). On the other hand, the same holds true for the thresholding algorithms. In the experiment described above we have used the same number of non-zero coefficients as was chosen by the WGL. This is not likely to be optimal for the thresholding algorithms since the obtimal number of non-zero coefficients usually depends on the sparsity of the signal (few coefficients for sparse signals and vice versa). Finally, we have not addressed the audio quality of the denoised sounds. In general, it is very hard to decide which algorithm sounds "the best" as this depends on the application and the subjective opinion of the listener. Let us however mention that there are indeed audiable differences between the algorithms. For the violin music in \Figref{fig:CleanNoisySpec}, the WGL does the best job of removing the noise, but at the price of a poor timbre of the resulting sound. As we include more coefficients at higher frequencies, the original timbre of the instrument improves together with an increase in noise.

\section{Conclusion}\label{Sec:6}
We have presented a new thresholding algorithm and proven an associated strong Jackson embedding under rather general conditions. The algorithm extends the classical greedy approach by incoorporating a weight function, which exploites the structure of the expansion coefficients. In particular, the algorithm applies to approximation in modulation spaces using Gabor frames. As an application we have considered the task of denoising music signals and compared the proposed method with the greedy thresholding approach and the WGL from social sparsity. The numerical experiments show that the proposed method can be used both for improving the time-frequency resolution and for reducing the RMS error compared to the to other algortihms. The experiments also show that the performance of the algorithm depends crucially on the choice of weight function, which should be adapted to the particular signal class under consideration.

\bibliographystyle{elsarticle-num}

\begin{thebibliography}{10}
\expandafter\ifx\csname url\endcsname\relax
  \def\url#1{\texttt{#1}}\fi
\expandafter\ifx\csname urlprefix\endcsname\relax\def\urlprefix{URL }\fi
\expandafter\ifx\csname href\endcsname\relax
  \def\href#1#2{#2} \def\path#1{#1}\fi

\bibitem{Schmidt07}
E.~Schmidt, \href{http://dx.doi.org/10.1007/BF01449770}{Zur {T}heorie der
  linearen und nichtlinearen {I}ntegralgleichungen}, Math. Ann. 63~(4) (1907)
  433--476.
\newblock \href {http://dx.doi.org/10.1007/BF01449770}
  {\path{doi:10.1007/BF01449770}}.
\newline\urlprefix\url{http://dx.doi.org/10.1007/BF01449770}

\bibitem{Ismagilov74}
R.~S. Ismagilov, Diameters of sets in normed linear spaces, and the
  approximation of functions by trigonometric polynomials, Uspehi Mat. Nauk
  29~(3(177)) (1974) 161--178.

\bibitem{Oskolkov78}
K.~I. Oskolkov, Polygonal approximation of functions of two variables, Mat. Sb.
  (N.S.) 107(149)~(4) (1978) 601--612, 639.

\bibitem{Petrushev1988}
P.~P. Petrushev, \href{http://dx.doi.org/10.1007/BFb0078887}{Direct and
  converse theorems for spline and rational approximation and {B}esov spaces},
  in: Function spaces and applications ({L}und, 1986), Vol. 1302 of Lecture
  Notes in Math., Springer, Berlin, 1988, pp. 363--377.
\newblock \href {http://dx.doi.org/10.1007/BFb0078887}
  {\path{doi:10.1007/BFb0078887}}.
\newline\urlprefix\url{http://dx.doi.org/10.1007/BFb0078887}

\bibitem{Grochenig2000}
K.~Gr\"ochenig, S.~Samarah,
  \href{http://dx.doi.org/10.1007/s003659910014}{Nonlinear approximation with
  local {F}ourier bases}, Constr. Approx. 16~(3) (2000) 317--331.
\newblock \href {http://dx.doi.org/10.1007/s003659910014}
  {\path{doi:10.1007/s003659910014}}.
\newline\urlprefix\url{http://dx.doi.org/10.1007/s003659910014}

\bibitem{DeVore1992}
R.~A. DeVore, B.~Jawerth, V.~Popov,
  \href{http://dx.doi.org/10.2307/2374796}{Compression of wavelet
  decompositions}, Amer. J. Math. 114~(4) (1992) 737--785.
\newblock \href {http://dx.doi.org/10.2307/2374796}
  {\path{doi:10.2307/2374796}}.
\newline\urlprefix\url{http://dx.doi.org/10.2307/2374796}

\bibitem{Butzer72}
P.~L. Butzer, K.~Scherer, Jackson and {B}ernstein-type inequalities for
  families of commutative operators in {B}anach spaces, J. Approximation Theory
  5 (1972) 308--342, collection of articles dedicated to J. L. Walsh on his
  75th birthday, III.

\bibitem{DeVore1993}
R.~A. DeVore, G.~G. Lorentz, Constructive approximation, Vol. 303 of
  Grundlehren der Mathematischen Wissenschaften [Fundamental Principles of
  Mathematical Sciences], Springer-Verlag, Berlin, 1993.
\newblock \href {http://dx.doi.org/10.1007/978-3-662-02888-9}
  {\path{doi:10.1007/978-3-662-02888-9}}.

\bibitem{Davis1997}
G.~Davis, S.~Mallat, M.~Avellaneda,
  \href{http://dx.doi.org/10.1007/BF02678430}{Adaptive greedy approximations},
  Constructive Approximation 13~(1) (1997) 57--98.
\newblock \href {http://dx.doi.org/10.1007/BF02678430}
  {\path{doi:10.1007/BF02678430}}.
\newline\urlprefix\url{http://dx.doi.org/10.1007/BF02678430}

\bibitem{Darken1993}
C.~Darken, M.~Donahue, L.~Gurvits, E.~Sontag,
  \href{http://doi.acm.org/10.1145/168304.168357}{Rate of approximation results
  motivated by robust neural network learning}, in: Proceedings of the Sixth
  Annual Conference on Computational Learning Theory, COLT '93, ACM, New York,
  NY, USA, 1993, pp. 303--309.
\newblock \href {http://dx.doi.org/10.1145/168304.168357}
  {\path{doi:10.1145/168304.168357}}.
\newline\urlprefix\url{http://doi.acm.org/10.1145/168304.168357}

\bibitem{DeVore1996}
R.~A. DeVore, V.~N. Temlyakov, \href{http://dx.doi.org/10.1007/BF02124742}{Some
  remarks on greedy algorithms}, Advances in Computational Mathematics 5~(1)
  (1996) 173--187.
\newblock \href {http://dx.doi.org/10.1007/BF02124742}
  {\path{doi:10.1007/BF02124742}}.
\newline\urlprefix\url{http://dx.doi.org/10.1007/BF02124742}

\bibitem{Gribonval2004}
R.~Gribonval, M.~Nielsen,
  \href{http://dx.doi.org/10.1007/s00041-004-8003-5}{Nonlinear approximation
  with dictionaries {I}. {D}irect estimates}, J. Fourier Anal. Appl. 10~(1)
  (2004) 51--71.
\newblock \href {http://dx.doi.org/10.1007/s00041-004-8003-5}
  {\path{doi:10.1007/s00041-004-8003-5}}.
\newline\urlprefix\url{http://dx.doi.org/10.1007/s00041-004-8003-5}

\bibitem{Kowalski2009a}
M.~Kowalski, \href{https://doi.org/10.1016/j.acha.2009.05.006}{Sparse
  regression using mixed norms}, Appl. Comput. Harmon. Anal. 27~(3) (2009)
  303--324.
\newline\urlprefix\url{https://doi.org/10.1016/j.acha.2009.05.006}

\bibitem{Kowalski2009b}
M.~Kowalski, B.~Torr{\'e}sani,
  \href{https://doi.org/10.1007/s11760-008-0076-1}{Sparsity and persistence:
  mixed norms provide simple signal models with dependent coefficients},
  Signal, Image and Video Processing 3~(3) (2009) 251--264.
\newblock \href {http://dx.doi.org/10.1007/s11760-008-0076-1}
  {\path{doi:10.1007/s11760-008-0076-1}}.
\newline\urlprefix\url{https://doi.org/10.1007/s11760-008-0076-1}

\bibitem{Plumbley2010}
M.~D. Plumbley, T.~Blumensath, L.~Daudet, R.~Gribonval, M.~E. Davies, Sparse
  representations in audio and music: From coding to source separation,
  Proceedings of the IEEE 98~(6) (2010) 995--1005.
\newblock \href {http://dx.doi.org/10.1109/JPROC.2009.2030345}
  {\path{doi:10.1109/JPROC.2009.2030345}}.

\bibitem{Siedenburg2011}
K.~{S}iedenburg, M.~{D}{\"o}rfler, {S}tructured sparsity for audio signals, in:
  Proc. of the 14th Int. Conference on Digital Audio Effects (DAFx-11), Paris,
  France, 2011.

\bibitem{Kowalski2009c}
M.~Kowalski, B.~Torr{\'e}sani,
  \href{https://hal.inria.fr/inria-00369577}{{Structured Sparsity: from Mixed
  Norms to Structured Shrinkage}}, in: R.~Gribonval (Ed.), {SPARS'09 - Signal
  Processing with Adaptive Sparse Structured Representations}, {Inria Rennes -
  Bretagne Atlantique}, Saint Malo, France, 2009.
\newline\urlprefix\url{https://hal.inria.fr/inria-00369577}

\bibitem{Kowalski2013}
M.~Kowalski, K.~Siedenburg, M.~D{\"o}rfler, Social sparsity! neighborhood
  systems enrich structured shrinkage operators, IEEE Transactions on Signal
  Processing 61~(10) (2013) 2498--2511.
\newblock \href {http://dx.doi.org/10.1109/TSP.2013.2250967}
  {\path{doi:10.1109/TSP.2013.2250967}}.

\bibitem{Tibshirani1996}
R.~Tibshirani, Regression shrinkage and selection via the lasso, Journal of the
  Royal Statistical Society (Series B) 58 (1996) 267--288.

\bibitem{Chen2001}
S.~S. Chen, D.~L. Donoho, M.~A. Saunders,
  \href{http://dx.doi.org/10.1137/S003614450037906X}{Atomic decomposition by
  basis pursuit}, SIAM Rev. 43~(1) (2001) 129--159.
\newblock \href {http://dx.doi.org/10.1137/S003614450037906X}
  {\path{doi:10.1137/S003614450037906X}}.
\newline\urlprefix\url{http://dx.doi.org/10.1137/S003614450037906X}

\bibitem{Siedenburg2012}
K.~{S}iedenburg, M.~{D}{\"o}rfler, {A}udio denoising by generalized
  time-frequency thresholding, in: Proceedings of the AES 45th Conference on
  Applications of Time-Frequency Processing, Helsinki, Finland, 2012.

\bibitem{Grochenig2001}
K.~Gr\"ochenig, \href{http://dx.doi.org/10.1007/978-1-4612-0003-1}{Foundations
  of time-frequency analysis}, Applied and Numerical Harmonic Analysis,
  Birkh\"auser Boston, Inc., Boston, MA, 2001.
\newblock \href {http://dx.doi.org/10.1007/978-1-4612-0003-1}
  {\path{doi:10.1007/978-1-4612-0003-1}}.
\newline\urlprefix\url{http://dx.doi.org/10.1007/978-1-4612-0003-1}

\bibitem{Christensen2016}
O.~Christensen, \href{http://dx.doi.org/10.1007/978-3-319-25613-9}{An
  introduction to frames and {R}iesz bases}, 2nd Edition, Applied and Numerical
  Harmonic Analysis, Birkh\"auser/Springer, [Cham], 2016.
\newblock \href {http://dx.doi.org/10.1007/978-3-319-25613-9}
  {\path{doi:10.1007/978-3-319-25613-9}}.
\newline\urlprefix\url{http://dx.doi.org/10.1007/978-3-319-25613-9}

\bibitem{DeVore98}
R.~A. DeVore, \href{http://dx.doi.org/10.1017/S0962492900002816}{Nonlinear
  approximation}, in: Acta numerica, 1998, Vol.~7 of Acta Numer., Cambridge
  Univ. Press, Cambridge, 1998, pp. 51--150.
\newblock \href {http://dx.doi.org/10.1017/S0962492900002816}
  {\path{doi:10.1017/S0962492900002816}}.
\newline\urlprefix\url{http://dx.doi.org/10.1017/S0962492900002816}

\bibitem{DeVore09}
R.~A. DeVore, \href{https://doi.org/10.1007/978-3-642-03413-8\_6}{Nonlinear
  approximation and its applications}, in: Multiscale, nonlinear and adaptive
  approximation, Springer, Berlin, 2009, pp. 169--201.
\newline\urlprefix\url{https://doi.org/10.1007/978-3-642-03413-8\_6}

\bibitem{Carro2007}
M.~a.~J. Carro, J.~A. Raposo, J.~Soria,
  \href{http://dx.doi.org/10.1090/memo/0877}{Recent developments in the theory
  of {L}orentz spaces and weighted inequalities}, Mem. Amer. Math. Soc.
  187~(877) (2007) xii+128.
\newblock \href {http://dx.doi.org/10.1090/memo/0877}
  {\path{doi:10.1090/memo/0877}}.
\newline\urlprefix\url{http://dx.doi.org/10.1090/memo/0877}

\bibitem{Stechkin1955}
S.~B. Stechkin, On absolute convergence of orthogonal series, Dokl. Akad. Nauk
  SSSR (N.S.) 102 (1955) 37--40.

\bibitem{Feichtinger89I}
H.~G. Feichtinger, K.~H. Gr\"ochenig,
  \href{http://dx.doi.org/10.1016/0022-1236(89)90055-4}{Banach spaces related
  to integrable group representations and their atomic decompositions. {I}}, J.
  Funct. Anal. 86~(2) (1989) 307--340.
\newblock \href {http://dx.doi.org/10.1016/0022-1236(89)90055-4}
  {\path{doi:10.1016/0022-1236(89)90055-4}}.
\newline\urlprefix\url{http://dx.doi.org/10.1016/0022-1236(89)90055-4}

\bibitem{Feichtinger89II}
H.~G. Feichtinger, K.~H. Gr\"ochenig,
  \href{http://dx.doi.org/10.1007/BF01308667}{Banach spaces related to
  integrable group representations and their atomic decompositions. {II}},
  Monatsh. Math. 108~(2-3) (1989) 129--148.
\newblock \href {http://dx.doi.org/10.1007/BF01308667}
  {\path{doi:10.1007/BF01308667}}.
\newline\urlprefix\url{http://dx.doi.org/10.1007/BF01308667}

\bibitem{Feichtinger85}
H.~G. Feichtinger, P.~Gr\"obner,
  \href{http://dx.doi.org/10.1002/mana.19851230110}{Banach spaces of
  distributions defined by decomposition methods. {I}}, Math. Nachr. 123 (1985)
  97--120.
\newblock \href {http://dx.doi.org/10.1002/mana.19851230110}
  {\path{doi:10.1002/mana.19851230110}}.
\newline\urlprefix\url{http://dx.doi.org/10.1002/mana.19851230110}

\bibitem{Feichtinger87}
H.~G. Feichtinger, \href{http://dx.doi.org/10.1002/mana.19871320116}{Banach
  spaces of distributions defined by decomposition methods. {II}}, Math. Nachr.
  132 (1987) 207--237.
\newblock \href {http://dx.doi.org/10.1002/mana.19871320116}
  {\path{doi:10.1002/mana.19871320116}}.
\newline\urlprefix\url{http://dx.doi.org/10.1002/mana.19871320116}

\bibitem{Borup07}
L.~Borup, M.~Nielsen, \href{http://dx.doi.org/10.1007/s00041-006-6024-y}{Frame
  decomposition of decomposition spaces}, J. Fourier Anal. Appl. 13~(1) (2007)
  39--70.
\newblock \href {http://dx.doi.org/10.1007/s00041-006-6024-y}
  {\path{doi:10.1007/s00041-006-6024-y}}.
\newline\urlprefix\url{http://dx.doi.org/10.1007/s00041-006-6024-y}

\bibitem{Voigtlaender2017}
F.~Voigtlaender, A.~Pein, \href{https://arxiv.org/abs/1702.03559}{Analysis vs.
  synthesis sparsity for $\alpha$-shearlets}, Ar{X}iv preprint
  arXiv:1702.03559v1.
\newline\urlprefix\url{https://arxiv.org/abs/1702.03559}

\bibitem{Ottosen2017}
E.~S. Ottosen, M.~Nielsen, \href{https://doi.org/10.1007/s00041-017-9546-6}{A
  characterization of sparse nonstationary {G}abor expansions}, J. Fourier
  Anal. Appl.\href {http://dx.doi.org/10.1007/s00041-017-9546-6}
  {\path{doi:10.1007/s00041-017-9546-6}}.
\newline\urlprefix\url{https://doi.org/10.1007/s00041-017-9546-6}

\bibitem{Nielsen2001}
R.~Gribonval, M.~Nielsen, Some remarks on non-linear approximation with
  {S}chauder bases, East J. Approx. 7~(3) (2001) 267--285.

\bibitem{Sondergaard2007}
P.~L. S{\o}ndergaard, {Finite discrete Gabor analysis}, Ph.D. thesis, Technical
  University of Denmark (2007).

\bibitem{Strohmer1998}
T.~Strohmer, \href{http://dx.doi.org/10.1007/978-1-4612-2016-9\_9}{Numerical
  algorithms for discrete {G}abor expansions}, in: Gabor analysis and
  algorithms, Appl. Numer. Harmon. Anal., Birkh\"auser Boston, Boston, MA,
  1998, pp. 267--294.
\newblock \href {http://dx.doi.org/10.1007/978-1-4612-2016-9\_9}
  {\path{doi:10.1007/978-1-4612-2016-9\_9}}.
\newline\urlprefix\url{http://dx.doi.org/10.1007/978-1-4612-2016-9\_9}

\bibitem{ltfatnote030}
Z.~Pr\r{u}\v{s}a, P.~L. S{\o}ndergaard, N.~Holighaus, C.~Wiesmeyr, P.~Balazs,
  \href{{http://dx.doi.org/10.1007/978-3-319-12976-1\_25}}{{The Large
  Time-Frequency Analysis Toolbox 2.0}}, in: M.~Aramaki, O.~Derrien,
  R.~Kronland-Martinet, S.~Ystad (Eds.), Sound, Music, and Motion, Lecture
  Notes in Computer Science, Springer International Publishing, 2014, pp.
  419--442.
\newblock \href {http://dx.doi.org/{10.1007/978-3-319-12976-1\_25}}
  {\path{doi:{10.1007/978-3-319-12976-1\_25}}}.
\newline\urlprefix\url{{http://dx.doi.org/10.1007/978-3-319-12976-1\_25}}

\bibitem{EBUSQAM}
E.~B. Union, \href{https://tech.ebu.ch/docs/tech/tech3253.pdf}{Sound quality
  assessment material: {R}ecordings for subjective tests ; {U}ser's handbook
  for the {EBU - SQAM} compact disc} (Sep 2008).
\newline\urlprefix\url{https://tech.ebu.ch/docs/tech/tech3253.pdf}

\end{thebibliography}

\end{document}